\documentclass{amsart}
\usepackage{amssymb}

\newcommand{\norm}[1]{\mbox{$\|#1\|$}}

\newcommand{\x}{\times}
\newcommand{\cs}{\mbox{$C^{*}$-algebra}}
\newcommand{\css}{\mbox{$C^{*}$-algebras}}

\newcommand{\tE}{\tilde{E}}
\newcommand{\tF}{\tilde{F}}
\newcommand{\tp}{\tilde{p}}

\newcommand{\C}{\mathbb{C}}

\newcommand{\R}{\mathbb{R}}

\newcommand{\ov}[1]{\mbox{$\overline{#1}$}}

\newcommand{\al}{\mbox{$\alpha$}}
\newcommand{\eps}{\mbox{$\epsilon$}}
\newcommand{\bt}{\mbox{$\beta$}}
\newcommand{\ga}{\mbox{$\gamma$}}
\newcommand{\Ga}{\mbox{$\Gamma$}}
\newcommand{\de}{\mbox{$\delta$}}
\newcommand{\De}{\mbox{$\Delta$}}

\newcommand{\la}{\mbox{$\lambda$}}

\newcommand{\si}{\mbox{$\sigma$}}

\newcommand{\mfH}{\mathfrak{H}}

\newcommand{\bc}{\begin{center}}

\newcommand{\ec}{\end{center}}
\newcommand{\be}{\begin{enumerate}}
\newcommand{\ee}{\end{enumerate}}

\newcommand{\beqn}{\begin{eqnarray}}
\newcommand{\eeqn}{\end{eqnarray}}
\newcommand{\beqns}{\begin{eqnarray*}}
\newcommand{\eeqns}{\end{eqnarray*}}
\newcommand{\bq}{\begin{quote}}
\newcommand{\eq}{\end{quote}}

\newcommand{\bi}{\begin{itemize}}
\newcommand{\ei}{\end{itemize}}
\newcommand{\bd}{\begin{description}}
\newcommand{\ed}{\end{description}}

\newcommand{\lan}{\mbox{$\langle$}}
\newcommand{\ran}{\mbox{$\rangle$}}

\theoremstyle{plain}
\newtheorem{theorem}{Theorem}[section]
\newtheorem{lemma}[theorem]{Lemma}
\newtheorem{proposition}[theorem]{Proposition}

\theoremstyle{definition}
\newtheorem{definition}[theorem]{Definition}

\theoremstyle{remark}

\numberwithin{equation}{section}
\begin{document}

\title{The analytic index for proper, Lie groupoid actions}

\author{Alan L. T. Paterson}
\address{Department of Mathematics, University of Mississippi, University,
Mississippi 38677}
\email{mmap@olemiss.edu}
\date{November 25, 1999.}

\begin{abstract}
Many index theorems (both classical and in noncommutative geometry)
can be interpreted in terms of a Lie groupoid acting properly on a manifold
and leaving an elliptic family of pseudodifferential operators invariant.
Alain Connes in his book raised the question of an index theorem in this
general context. In this paper, an analytic index for many such situations is
constructed. The approach is inspired by the classical families theorem of
Atiyah and Singer, and the proof generalizes,
to the case of proper Lie groupoid actions,
some of the results proved for proper locally compact group actions by N. C. Phillips.
\end{abstract}

\maketitle

\section{Introduction}

The objective of this paper is to prove the following theorem.

\begin{theorem}  \label{maintheorem}
Let $G$ be a Lie groupoid for which the unit space $G^{0}$ is a proper,
$G$-compact $G$-space. Let $(X,p)$ be a
$G$-manifold which is a fiber bundle over $G^{0}$ with
compact smooth manifold $Z$ as fiber and with structure group $\text{Diff}(Z)$.
Let $\tE,\tF$ be smooth $G$-vector bundles over $X$ and
$D= \{D^{u}\}:C^{0,\infty}(X,\tE)\to C^{0,\infty}(X,\tF)$ be an invariant,
continuous family of elliptic pseudodifferential operators on the fibers
$X^{u}$ of $X$.   Then $D$ determines a $(\C,C_{red}^{*}(G))$-Kasparov module
whose $K_{0}(C_{red}^{*}(G))$-class is the analytic index of $D$.
\end{theorem}

The motivation for this result comes from noncommutative geometry. In his
book (\cite[p.151]{Connesbook}), Alain Connes comments that a number of index
theorems both in classical and in noncommutative geometry are all special cases of
the same index theorem for $G$-invariant elliptic operators $D$ on a proper
$G$-manifold $X$, $G$ being a Lie groupoid. In this situation, the index of $D$
belongs to $K_{0}(C^{*}(G))$. The present paper will establish this under
certain (commonly satisfied) conditions. The approach to the analytic index of
$D$ in this paper is based on an adaptation of the equivariant Atiyah-Singer
index theorem for families. There are, of course, other approaches to the
analytic index in the context of groupoids, notably those using
quasi-isomorphisms and the ``deformation'' approach using the tangent groupoid
(e.g. \cite{Connesbook,Lands,LN,Mont,MontP,NWX,Paterson}). However, these
approaches apply (at the present time) only to the case where $X=G$, and in
particular, do not cover the classical families index theorem. Further, the
action of $G$ on itself is rather special: for example, that action is free. In
the present paper, no freeness assumptions of the action of $G$ on $X$ are
made.

Another merit in the approach here is that it relates more readily to
pseudodifferential operators and Fredholm theory and the well-established,
detailed, classical theory. It also gives rise to a groupoid equivariant
K-theory for $TX$ and the prospect of defining a topological index in the same
spirit as that of the classical families index theorem.  This will be
discussed elsewhere.

The author believes that the proof of this paper can be adapted to give the
analytic index in the full generality of \cite[p.151]{Connesbook} and plans to
discuss this elsewhere. The restrictions in the present paper are largely a
matter of convenience, since they ennable one to use existing theories to
shorten and simplify the proof. In particular, the space $X$ is assumed, as in
the work of Atiyah and Singer (\cite{AS4}), to be a manifold with compact fiber
over the unit space $G^{0}$ of $G$. (The base space $G^{0}$ is not, however,
assumed to be compact, but rather just $G$-compact and proper as a $G$-space.)
This ennables one to appeal, for example,
to the results on pseudodifferential operators used in \cite{AS4}. Another
advantage of this framework is that it gives rise to a Hilbert $G$-bundle, and
this allows the techniques used by N. C. Phillips in his development of
equivariant K-theory to be adapted to produce the required Kasparov module.
Further, we have used the reduced $\cs$
$C^{*}_{red}(G)$ rather than $C^{*}(G)$, thus avoiding having to use the
disintegration theorem for locally compact groupoids. So the index of $D$
constructed here lies in $K_{0}(C_{red}^{*}(G))$.

The second section collects, for the convenience of the reader, some facts and
definitions about locally compact and Lie groupoids, while the third
discusses $G$-spaces and $G$-manifolds.  The fourth section discusses
families of elliptic pseudodifferential operators invariant under the action
of a Lie groupoid $G$. Such a family $D$ gives an invariant Fredholm morphism
on a certain Hilbert bundle, and the final section shows that associated with
this bundle is a Hilbert module, whose bounded and compact module maps
correspond exactly to the invariant bounded and compact  morphisms on the
Hilbert bundle.  This section is modelled on the approach of N. C. Phillips
in his book \cite{Phillips} -- Phillips effectively deals with the case where
$G$ is a transformation group groupoid. So the Fredholm morphism $D$ translates
over to a Fredholm map on a Hilbert module over $C_{red}^{*}(G)$, which, by a
standard process, gives a Kasparov module whose $KK$-class is the analytic
index of $D$.

The author is grateful to the organizers for the opportunity to speak at the
groupoid conference of the Colorado JSRC meeting of 1999, and for their warm
hospitality. He would also like to express his thanks to N. C. Phillips for
helpful conversations about equivariant K-theory.

\section{Preliminaries on groupoids}

If $X$ is a locally compact Hausdorff space, then $\mathcal{C}(X)$ is the
family of compact subsets of $X$, $C(X)$ is the algebra of bounded continuous
complex-valued functions on $X$, and $C_{c}(X)$ is the subalgebra of $C(X)$
whose elements have compact support. For a smooth manifold $X$, we define
$C^{\infty}(X)$ to be the algebra of $C^{\infty}-$ complex-valued functions on
$X$, and $C_{c}^{\infty}(X)$ to be the algebra of functions in $C^{\infty}(X)$
with compact support.

A {\em groupoid} is most simply defined as a small category with inverses.
Spelled out axiomatically, a groupoid
is a set $G$ together with a subset $G^{2} \subset G\times G$, a product map
$m:G^{2}\to G$, where we write $m(a,b)=ab$,
and an inverse map $i:G\to G$, where we write $i(a)=a^{-1}$
and  where $(a^{-1})^{-1}=a$,  such that:
\begin{enumerate}
\item if $(a,b), (b,c)\in G^{2}$, then $(ab,c), (a,bc)\in G^{2}$
and
\[     (ab)c=a(bc);        \]
\item $(b,b^{-1})\in G^{2}$ for all $b\in G$, and if $(a,b)$
belongs to $G^{2}$, then
\[     a^{-1}(ab) = b \hspace{.2in} (ab)b^{-1} = a. \]
\end{enumerate}
We define the {\em range} and {\em source} maps $r:G\to G^{0}$, $s:G\to G^{0}$
by setting $r(x)=xx^{-1}, s(x)=x^{-1}x$.
The {\em unit space} $G^{0}$ is defined to be $r(G)(=s(G))$, or equivalently,
the set of idempotents $u$ in $G$. The maps $r,s$ fiber the groupoid $G$ over
$G^{0}$ with fibers $\{G^{u}\}, \{G_{u}\}$, so that $G^{u}=r^{-1}(\{u\})$ and
$G_{u}=s^{-1}(\{u\})$. Note that $(x,y)\in G^{2}$ if and only if $s(x)=r(y)$.

For detailed discussions of groupoids (including locally compact and Lie
group\-oids as below), the reader is referred to the books
\cite{Lands,Mackenzie,Paterson,rg}. Important examples of groupoids are given
by transformation group groupoids and equivalence relations.

For the purposes of this paper, a {\em locally compact groupoid} is a groupoid
$G$ which is also a second countable, locally compact Hausdorff space for which
multiplication and inversion are continuous. More generally, there is a notion
of {\em locally compact groupoid} which does not require the Hausdorff
condition. These arise quite often in practice, and a detailed discussion of
them is given in \cite{Paterson}. The details of the paper can be made to work
in the general case since local arguments suffice. However, for convenience,
the locally compact groupoids considered in the paper are assumed to be
Hausdorff.

Let $G$ be a locally compact groupoid. Note that $G^{2},G^{0}$ are closed
subsets of $G\x G, G$ respectively. Further, since $r,s$ are continuous, every
$G^{u}, G_{u}$ is a closed subset of $G$.

For analysis on a locally compact groupoid $G$, it is essential to have
available a {\em left Haar system}.  This is the groupoid version of a {\em
left Haar measure}, though unlike left Haar measure on a locally compact
group, such a system may not exist and if it does, it will not usually be
unique.  However, in many cases, such as in the  Lie groupoid case below,
there is a natural choice of left Haar system.

A left Haar system on a locally compact groupoid $G$ is a family of measures
$\{\la^{u}\}$ $(u\in G^{0})$, where each $\la^{u}$ is a
positive regular Borel measure on the locally compact Hausdorff space
$G^{u}$, such that the following three axioms are satisfied:
\be
\item the support of each $\la^{u}$ is the whole of $G^{u}$;
\item for any $f\in C_{c}(G)$, the function $f^{0}$, where
\[   f^{0}(u)=\int_{G^{u}}f\,d\la^{u},   \]
belongs to $C_{c}(G^{0})$;
\item for any $g\in G$ and $f\in C_{c}(G)$,
\begin{equation} \label{eq:invar}
\int_{G^{s(g)}}f(gh)\,d\la^{s(g)}(h) = \int_{G^{r(g)}}f(h)\,d\la^{r(g)}(h).
\end{equation}
\ee
The existence of a left Haar system on $G$ has topological consequences for
$G$ -- it entails that both $r,s:G\to G^{0}$ are open maps
(\cite[p.36]{Paterson}).  For each $u\in G^{0}$, the measure $\la_{u}$
on $G_{u}=(G^{u})^{-1}$ is given by: $\la_{u}(A)=\la^{u}(A^{-1})$.

Jean Renault showed (\cite{rg}) that $C_{c}(G)$, the space of continuous,
complex-valued functions on $G$ with compact support, is a convolution
$^{*}$-algebra with product given by:
\begin{equation}
f_{1}*f_{2}(g) =
\int_{G^{r(g)}}f_{1}(h)f_{2}(h^{-1}g)\,d\la^{r(g)}(h). \label{eq:convo}
\end{equation}
The involution on $C_{c}(G)$ is the map $f\to f^{*}$ where
\begin{equation}
 f^{*}(g)=\ov{f(g^{-1})}.  \label{eq:inv}
\end{equation}

As in the case of a locally compact group, there is a reduced $\cs$
$C_{red}^{*}(G)$ (and a universal $\cs$ $C^{*}(G)$) for the locally compact
groupoid $G$. The reduced $\cs$ can be defined as follows.  For each $u\in
G^{0}$, we first define
a representation $\pi_{u}$ of $C_{c}(G)$
on the Hilbert space $L^{2}(G,\la_{u})$.  To this end,
regard $C_{c}(G)$ as a dense subspace of $L^{2}(G,\la_{u})$ and define
for $f\in C_{c}(G), \xi\in C_{c}(G)$,
\begin{equation}
     \pi_{u}(f)(\xi)(g)=f*\xi\in C_{c}(G).       \label{eq:red}
\end{equation}
The reduced $\cs$-norm on $C_{c}(G)$ can then (\cite[p.108]{Paterson})
be defined by:
\[        \norm{f}_{red}=\sup_{u\in G^{0}}\norm{\pi_{u}(f)}.       \]
The groupoid $\cs$ $C^{*}(G)$ is the completion of $C_{c}(G)$
under its largest $C^{*}$-norm.
In the paper, we will concentrate (for reasons of technical simplicity) on the
reduced $\cs$ of $G$.

A locally compact groupoid $G$ is called a {\em Lie groupoid}
(\cite{Mackenzie,Paterson,Pra66,Pra67}) if $G$ is a
manifold such that:
\be
\item $G^{0}$ is a submanifold of $G$;
\item the maps $r,s:G\to G^{0}$ are submersions;
\item the product and inversion maps for $G$ are smooth.
\ee
Note that $G^{2}$ is naturally a submanifold of $G\x G$ and every $G^{u},
G_{u}$ is a submanifold of $G$.  (See \cite[pp.55-56]{Paterson}.)  Note that
the maps $r,s$ are open maps. The dimension of all of the $G^{u}, G_{u}$ is the
same, and will be denoted by $l$.

There is a basis for the topology of  a Lie groupoid that proves useful.
Since $r$ is a submersion, it is locally equivalent to a projection.  So for
each $x\in G$, there exists an open neighborhood $V$ of $x$ and a
diffeomorphism $\psi:V\to r(V)\x W$ where $W$ is an open subset of $\R^{l}$,
which is fiber preserving in the sense that we can write
$\psi(x)=(r(x),\phi(x))$ for some function $\phi$.  So if $u\in r(V)$, then
the restriction of $\psi$ to $G^{u}\cap V$ is a diffeomorphism onto $W$.
We will call such an open set $V$ a {\em basic} open set and write $V\sim
r(V)\x W$.

A. Connes (\cite[p.101]{Connesbook}) discusses convolution on
$C_{c}^{\infty}(G)$ in terms of density bundles, and this is canonical.
However, since the bundles involved are trivial, we can replace the densities
by a left Haar system $\{\la^{u}\}$ which is {\em smooth} in the sense that in
terms of a basic open set $V$, the map $(u,w)\to d\la^{u}/d\la_{W}(u,w)$ is
$C^{\infty}$ and strictly positive on $r(V)\x W$. (Here, $\la_{W}$ is the
restriction of Lebesgue measure on $\R^{l}$ to $W$.) For more details about
this, see \cite[2.3]{Paterson}. All such smooth left Haar systems are
equivalent in the obvious sense, and give the same
$C_{red}^{*}(G)$ (and the same $C^{*}(G)$). For the rest of the paper,
$\{\la^{u}\}$ will be a fixed smooth left Haar system on the Lie groupoid $G$.

\section{$G$-spaces and $G$-manifolds}

For the purposes of groupoid theory, we require the notions of proper $G$-spaces
and $G$-manifolds (e.g. \cite{MuW}, \cite[p.151]{Connesbook}).

Let $X,Y$ be locally compact, second countable topological spaces and
$p:X\to Y$ be a continuous, open surjection.  We will then say that
$(X,p)$ is {\em fibered over $Y$}. We will be particularly concerned with the
case where $Y=G^{0}$, the unit space of a locally compact groupoid $G$. The
space fibers of $X$ are then the sets $X^{u}=p^{-1}(\{u\})$ ($u\in G^{0}$).

If $(X_{i},p_{i})$
($i=1,2$) are fibered over $G^{0}$,
then the fibered product $(X_{1}*X_{2},\tilde{p})$ is also fibered
over $G^{0}$, where
\[  X_{1}*X_{2}= \{(x,y)\in X_{1}\x X_{2}: p_{1}(x)=p_{2}(y)\},  \]
and $\tilde{p}(x,y)=p_{1}(x)=p_{2}(y)$.

Let $G*X$ be the fibered product of $(G,s)$ with $(X,p)$.
The pair $(X,p)$ is defined to be a {\em
$G$-space} if there is given a continuous map $m:G*X\to X$, where we write
$gx$ for $m(g,x)$, such that for $(g,x)\in G*X$ and $(h,g)\in G^{2} (=G*G)$, we
have:
\be
\item $p(gx)=r(g)$;
\item $h(gx)=(hg)x$;
\item $g^{-1}(gx)=x$.
\ee

Of course, the pair $(G,r)$ is a $G$-space under the groupoid multiplication.
We note here that the unit space $G^{0}$ of $G$, as well as $G$ itself,
is a $G$-space.
Here, the map $p$ is just the identity map, and the action is given by:
     \[                 g.s(g)=r(g).                      \]

If $(X_{i},p_{i})$
($i=1,2$) are $G$-spaces,
then the fibered product $(X_{1}*X_{2},\tilde{p})$
is also a $G$-space under the action: $(g,x,y)\to (gx,gy)$ $(g\in G^{p(x)})$.
For a $G$-space $X$, let $G*_{r} X=\{(g,x): r(g)=p(x)\}$
be the fibered product of the $G$-spaces $(G,r)$ and $(X,p)$.

Let $(X,p)$ be a $G$-space.  Then for each $g\in G$, the map
$l_{g}:X^{s(g)}\to X^{r(g)}$, where $l_{g}(x)=gx$,
is a homeomorphism, with inverse $l_{g^{-1}}$.
Note also that if $u\in G^{0}$, then $l_{u}$ is the identity map on $X^{u}$.
Also, associated with the left action of $G$ on
$X$ is a left action $g\to L_{g}$ of $G$
on the bundle of algebras $C_{c}(X^{u})$
$(u\in G^{0})$.  For this, we define
$L_{g}:C_{c}(X^{s(g)})\to C_{c}(X^{r(g)})$ by:
\[          L_{g}f(x)=f(g^{-1}x)\hspace{.1in} (x\in X^{s(g)}).        \]
Clearly, each $L_{g}$ is an isomorphism of $^{*}$-algebras, and both
$(L_{g})^{-1}=L_{g^{-1}}$
and $L_{gh}=L_{g}L_{h}$ for $(g,h)\in G^{2}$.

We now discuss properness and $G$-compactness for locally compact groupoid
actions (on $G$-spaces). Proper actions for locally compact groups have been
investigated by R. S. Palais (\cite{Palais}) and N. C. Phillips
(\cite{Phillips}). The basic results for proper locally compact group actions
extend to the case of proper locally compact groupoid actions. Specifically,
the action of a locally compact groupoid $G$ on a $G$-space $X$ gives the orbit
equivalence relation on $X$. Further, if $X$ is a proper $G$-space, then the
space $X/G$ is locally compact Hausdorff in the quotient topology, and the
quotient map is open (\cite{MuRW,MuW}).

The $G$-space $X$ is called a {\em proper} $G$-space if the (continuous) map
$\al$, given by $(g,x)\to (gx,x)$, from $G*X$ into $X\x X$, is a proper map, i.e.
if $\al^{-1}(C)$ is compact in $G*X$ whenever $C$ is compact in $X\x X$.  It
is well-known, and easy to check, that $G$ itself as a $G$-space is proper.
The $G$-space $X_{1}*X_{2}$ is proper if both $G$-spaces
$X_{1},X_{2}$ are.

The $G$-space $G^{0}$ is not usually proper (even when $G$ is a Lie group).
However, if $G^{0}$ is a proper $G$-space -- and this is an assumption of
Theorem~\ref{maintheorem} -- then {\em every} $G$-space is proper.

The $G$-space $X$ is called {\em
$G$-compact} (cf. \cite[p.23]{Phillips})
if the quotient space $X/G$ is compact. Following the proof of
\cite[Lemma, p.23-24]{Phillips}, we obtain that {\em if $X$ is a {\em proper}
$G$-manifold, then $X$ is $G$-compact if and only if there exists a compact
subset $B$ of $X$ such that $X=GB$.}

One consequence of the existence of a $G$-compact, proper $G$-space $X$ is
that {\em $G^{0}$ as a $G$-space is also $G$-compact.} For then there is a
canonical, surjective, continuous map from $X/G$ onto $G^{0}/G$, so that the
latter is compact if the former is. In this paper, we will assume that $G^{0}$
is $G$-compact. (This implies that the $G$-space $X$ of the theorem of this
paper is itself $G$-compact.)

We now discuss invariant averaging and $G$-partition of
unities for a proper $G$-space.  The discussion generalizes the approach of
Phillips (\cite[Chapter 2]{Phillips}).  (For earlier special cases of this,
see \cite{CoMos,Kasp1}.)

To deal with the varying fibers that arise in groupoid contexts and also with
the vector bundle context, we use the following technical lemma. This is a
variant of the lemma \cite[Lemma C.0.2]{Paterson} (the main idea of which goes
back to Renault and Connes). The proof of the lemma is the same {\em mutatis
mutandis} as that of \cite[Lemma C.0.2]{Paterson}.
\newcommand{\mcE}{\mathcal{E}}

\begin{lemma}        \label{lemma:cont}
Let $(X,p)$ be a proper $G$-space
and $(\mcE,\rho)$ be a vector bundle over $X$.  Let
$f':G*_{r} X\to \mcE$ be
continuous and such that $f'((g,x))\in \mcE^{x}$
for each $x\in X$.
Suppose further that for each $C\in \mathcal{C}(X)$ and with
\[  W_{C}=\{(g,x)\in G*_{r} X: x\in C\},        \]
the restriction $f'\mid_{W_{C}}\in C_{c}(W_{C},\mcE)$.  Then the
fiber preserving map $F:X\to \mcE$, where
\begin{equation}      \label{eq:f'F}
F(x)=\int f'(g,x)\,d\la^{p(x)}(g)
\end{equation}
is continuous on $G$.
\end{lemma}

The scalar version of the above lemma is the case
where $\mcE$ is the trivial bundle $X\x \C$: in that case, the
function $f'$ is scalar valued.  If, in addition,
$X=G$, then the result is equivalent to \cite[Lemma C.0.2]{Paterson}.

The scalar version of Lemma~\ref{lemma:cont} ennables us, in the presence of
properness, to average a $C_{c}(X)$-function into an invariant continuous
function on $X$. The same argument applies more generally to sections of Banach
$G$-bundles, cf. \cite[Lemma 2.4]{Phillips}.

A function $f:X\to \C$ is called {\em invariant} if
for all $x\in X$ and all $h\in G^{p(x)}$, we have
\[       f(h^{-1}x)=f(x).           \]

\begin{proposition}    \label{prop:etainvar}
Let $X$ be a proper $G$-space and $\xi\in C_{c}(X)$.  Then
the function $\eta:G\to \C$, where
\begin{equation}     \label{eq:etainvar}
\eta(x)=\int_{G^{p(x)}}\xi(g^{-1}x)\,d\la^{p(x)}(g).
\end{equation}
is invariant and continuous.
\end{proposition}
\begin{proof}
Apply the scalar version of Lemma~\ref{lemma:cont} with
\[  f'(g,x)=\xi(g^{-1}x)        \]
to obtain that $\eta$ is well-defined and continuous.
For invariance, we have, using (\ref{eq:invar}):
\beqns
\eta(h^{-1}x) &=& \int\xi(g^{-1}h^{-1}x)\,d\la^{s(h)}(x)  \\
&=& \int\xi((hg)^{-1}x)\,d\la^{s(h)}(x) \\
&=& \int\xi(k^{-1}x)\,d\la^{p(x)}(k)  \\
&=& \eta(x).
\eeqns
\end{proof}

Let $X$ be a proper $G$-space and $\{U_{\al}\}$ be a collection of open
subsets of $X$. A {\em $G$-partition of unity of $X$ subordinate to
$\{U_{\al}\}$} (cf. \cite[p.25]{Phillips}) is a collection of continuous
non-negative functions
$f_{\ga}$ with compact supports such that for every $x\in X$, we
have
\begin{equation}  \label{eq:sumfga}
\sum_{\ga}\int_{G^{p(x)}}f_{\ga}(g^{-1}x)\,d\la^{p(x)}(g)=1
\end{equation}
where the sum is locally finite.

\begin{proposition}      \label{prop:aver}
Let $X$ be a proper $G$-space and $\{U_{\al}\}$ be a collection of open
subsets of $X$ such that the sets $GU_{\al}$ cover $X$.  Then there exists a
$G$-partition of unity subordinate to the collection $\{U_{\al}\}$.
\end{proposition}
\begin{proof}
We follow the proof of the group version of the result of Phillips
in \cite[Lemma
2.6]{Phillips}.
Let $\pi:X\to X/G$ be the quotient map. Since $X$ is a proper $G$-space, the
space $X/G$ is a locally compact Hausdorff, second countable space and $\pi$
is open. The group argument of Phillips then goes through to give continuous
functions $u_{\ga}:X\to [0,1]$ ($\ga\in S$)
with compact supports and the appropriate
local finiteness properties, and for which the function $f_{1}$ on $X$ given
by
\[  f_{1}(x)=\sum_{\ga\in S} \int u_{\ga}(g^{-1}x)\,d\la^{p(x)}(g)  \]
is (by Proposition~\ref{prop:etainvar}) strictly positive and continuous.
One then takes
$f_{\ga}=u_{\ga}/f_{1}$.
\end{proof}

\begin{proposition}       \label{prop:propc}
Let $X$ be a proper, $G$-compact $G$-space.  Then
there exist a non-negative
function $c\in C_{c}(X)$
such that for each $x\in X$, we have
\begin{equation}
\int_{G^{p(x)}}c(g^{-1}x)\,d\la^{p(x)}(g)=1.    \label{eq:cgm1}
\end{equation}
\end{proposition}
\begin{proof}
Let $Q:X\to X/G$ be the quotient map.  Since $Q$ is open and $X/G$ is
compact, there exists a compact subset $C$ of $X$ such that $Q(C)=X/G$.  Let
$\xi$ be in $C_{c}(X)$ such that $0\leq \xi\leq 1$ and
$\xi(C)=\{1\}$.  Define $\eta$ as in (\ref{eq:etainvar}) and take $c=\xi/\eta$.
\end{proof}

The smooth version of a $G$-space is a {\em $G$-manifold} $(X,p)$. For this,
$G$ is now a Lie groupoid and $X$ is a smooth manifold. The map $p:X\to G^{0}$
is required to be a submersion. Then each $X^{u}$ is a submanifold of $X$ of
fixed dimension $j$. As in the case of $G$, there is a basis for the topology
of $X$ consisting of {\em basic} open sets $V\sim p(V)\x Y$ where $Y$ is an
open subset of $\R^{j}$.

Since the map $s\x p:G\x X\to G^{0}\x G^{0}$ is also a submersion, the space
$G*X=(s\x p)^{-1}(\De)$ is a submanifold of $G\x X$, where $\De$ is the
diagonal $\{(u,u): u\in G^{0}\}$.
Under the above assumptions, the pair $(X,p)$ is called a {\em $G$-manifold} if
it is a $G$-space such that the multiplication map $m:G*X\to X$ is smooth and inversion is a diffeomorphism.

Note that the pairs $(G,r), (G^{0},id)$ are $G$-manifolds. Also, if
$(X_{i},p_{i})$ $(i=1,2)$ are $G$-manifolds, then $(X_{1}*X_{2},\tilde{p})$ is
also a $G$-manifold.

If $(X,p)$ is a $G$-manifold, then for each $g\in G$, the map
$l_{g}:X^{s(g)}\to X^{r(g)}$ is a diffeomorphism, with inverse $l_{g^{-1}}$.
The (restricted) map
$L_{g}:C_{c}^{\infty}(X^{r(g)})\to C_{c}^{\infty}(X^{s(g)})$ is an isomorphism
of $^{*}$-algebras.

In this paper, hermitian metrics are assumed conjugate linear in the
first variable and linear in the second.
It will follow from Proposition~\ref{prop:hermG}, with $\tE=TX$,
that there is a $G$-isometric  hermitian metric
$\lan,\ran$ on the (complex) tangent bundle $TX$, the complexification of the
vector bundle of
vectors tangent along the fibers $X^{u}$'s.  (So in this paper, $TX$
is {\em not} the (complex) tangent bundle to the manifold $X$.) This gives
in the standard way a family of smooth measures $\mu^{u}$ on the $X^{u}$'s.
So in terms of local
coordinates, we have $d\mu^{u}=\mid\det\,g_{ij}^{u}\mid^{1/2}\,dx$,
where the
hermitian metric $\lan ,\ran$ restricts over $X^{u}$ to $\lan ,\ran^{u}$,
and the latter is given locally by $\sum g_{ij}^{u}\,dx_{i}\otimes dx_{j}$.
The $G$-isometric property of $\lan ,\ran$ just says that the differential
$(\ell_{g})_{*}:TX^{s(g)}\to TX^{r(g)}$ is an isometry.
This in turn translates to the $G$-invariance of the
$\mu^{u}$'s: for $f\in C_{c}(X^{r(g)})$, we have
\begin{equation}
\int _{G^{s(g)}}f(gx)\,d\mu^{s(g)}(x)=\int_{G^{r(g)}}f(y)\,d\mu^{r(g)}(y).
\label{eq:invarmu}
\end{equation}
The counterparts to (i) and (ii) in the definition of a {\em left Haar system}
(\S 2) also hold.

Note that the $C^{\infty}$-versions of Lemma~\ref{lemma:cont},
Proposition~\ref{prop:etainvar}, Proposition~\ref{prop:aver} and
Proposition~\ref{prop:propc} hold for a $G$-manifold $X$
-- the proof modifications are simple.

\section{Pseudodifferential operators on $G$-manifolds}

For the rest of the paper, we will assume that
{\em $G$ is a Lie groupoid, $G^{0}$ is a proper, $G$-compact, $G$-manifold, and
$(X,p)$ is a $G$-manifold which is, in addition, a fiber bundle over
$G^{0}$ with smooth compact fiber $Z$ and structure group $\text{Diff}(Z)$.}
Consideration of this situation is motivated by the Atiyah-Singer families
theorem, which can be interpreted as the case in which $X$ is compact and
$G=G^{0}$, a groupoid of units. (However, in the Atiyah-Singer situation, the
base space $G^{0}$ is not assumed to be a manifold, and the fiber bundle is not
assumed to be a manifold but rather a {\em manifold over $G^{0}$}. The
``continuous family'' version of the theorem of this paper requires continuous
family groupoids (\cite{Patcont}) and will be considered elsewhere.)

Note that the existence of such a $G$-manifold imposes conditions on the Lie
groupoid $G$.  For example, since the $G$-action is proper and every $X^{u}$
is compact, it follows that
for every $u,v\in G^{0}$, the set $\{(g,x)\in G*X: (gx,x)\in X^{v}\x X^{u}\}$
is compact.  So the sets $G_{u}\cap G^{v}$ are compact.  In particular, the
isotropy groups of $G$ are compact.   It is left to the reader to check that
$X$ is $G$-compact if and only if $G^{0}$ is $G$-compact and $X$ is proper if
and only if $G^{0}$ is proper. (So in theorem of this paper, $X$ is
automatically proper and $G$-compact.) Note that by replacing $\mu^{u}$ by
$\mu^{u}/\norm{\mu^{u}}$, we can suppose that each $\mu^{u}$ is a probability
measure.

We now discuss invariant
families of pseudodifferential operators on $X$.
It is natural to impose the condition on the symbols of such operators that
they belong to one of the H\"{o}rmander class $L^{m}_{\rho,\de}$ (as Kasparov
does in \cite{Kasp1}). However, as Atiyah and Singer comment
(\cite[p.508]{AS1}), {\em for our purposes, any one of these classes would be
equally good}. The reason is that for the purposes of the analytic index, the
class of operators (associated with the class of symbols taken) gets closed up
in a $\cs$, and this closure is the same for reasonable choices of starting
class. We use the class of pseudodiferential operators in \cite{AS1}, and for
the convenience of the reader, will now recall some of the details about these
operators (\cite{AS1},\cite[Ch. 4]{Wells}).

Let $B$ be an open subset of
some $\R^{N}$ and $m\in \R$. Then the space
$\tilde{S}^{m}(B\x \R^{N};B\x \C))$ of symbols is
the set of smooth functions $a:B\x \R^{N}\to \C$ such that given
multiindices $\al, \bt$ and a compact subset $K\subset B$, there exists a
constant $C$ such that
\begin{equation}  \label{symbol}
\mid\partial^{\bt}_{x}\partial^{\al}_{\xi}a(x,\xi)\mid
\leq C(1+\mid\xi\mid)^{m-\mid\al\mid}
\end{equation}
for all $x\in K,\xi\in \R^{N}$.

The symbols $a$ that we will be considering
are required to satisfy two other properties.  Firstly, it is
required that for all $x\in B$ and all $\xi\neq 0$ in $\R^{N}$,
the principle symbol
\begin{equation}   \label{eq:spa}
\si(a)(x,\xi)= \lim_{k\to \infty}\frac{a(x,k\xi)}{k^{m}}  
\end{equation}
exists.  Clearly, $\si(a):B\x (\R^{N}\sim \{0\})\to  \C$.
Secondly, we also require that for any ``cut-off'' function
$\psi\in C^{\infty}(\R^{N})$, i.e.
a $C^{\infty}$-function on $\R^{N}$ which is $0$ near $\xi=0$
and $1$ outside the unit ball, we have
\[   a(x,\xi) - \psi(\xi)\si(a)\in \tilde{S}^{m-1}(B\x \R^{N};B\x \C). \]
The set of symbols in $\tilde{S}^{m}(B\x\R^{N};B\x\C)$
which also satisfy these two
properties is denoted by $S^{m}(B\x\R^{N};B\x\C)$.

More generally, the space $S^{m}(B\x \R^{N};B\x\C^{k})$ is defined in the
obvious way by considering functions $a:B\x \R^{N}\to \C^{k}$ whose components
are in $S^{m}(B\x \R^{N};B\x\C)$. The space $S^{m}(B\x \R^{N};B\x\C^{k})$ is a
Fr\'{e}chet space using as seminorms the $\norm{.}_{m,B,\al,\bt,K}$'s, where
$\norm{.}_{m,B,\al,\bt,K}$ is the smallest constant $C$ for which the
inequality of (\ref{symbol}) holds.

Any $a\in S^{m}(B\x\R^{N};B\x\C^{k})$
determines a pseudodifferential operator
$a(x,D_{x}):C_{c}^{\infty}(B;\C^{k})\to
C^{\infty}(B;\C^{k})$, 
where for $f\in C_{c}^{\infty}(B;\C^{k})$,
\begin{equation}  \label{eq:ayD}
a(x,D_{x})f(x)=
(2\pi)^{-N}\int_{\R^{N}}\,e^{ix.\xi}a(x,\xi)\hat{f}(\xi)\,d\xi.
\end{equation}
Define the seminorm $\norm{,}_{m,B,\al,\bt,K}$ on the operators
$a(x,D_{x})$ by:
\[  \norm{a(x,D_{x})}_{m,B,\al,\bt,K}
=\norm{a}_{m,B,\al,\bt,K}.  \]

Regarding $B\x \R^{N}$ as $T^{*}B$ and replacing $B\x \C^{k}$
by a smooth complex vector
bundle, the above definitions generalize in the standard way to the setting of
smooth complex vector bundles $E,F$ over a
manifold $M$. (This is because the class of pseudodifferential operators
above are invariant under diffeomorphisms.)
We will take $M$ to be compact but this is not essential.
In this setting, an
element of $S^{m}(M;E,F)$ is just a section
of the bundle $Hom(E,F)$ pulled back over the cotangent bundle
$T^{*}M$. Locally, $Hom(E,F)$ is of the form $B\x\C^{k}$ but $\C^{k}$ is
realized as the space of $q\x p$-matrices, where $p,q$ are the ranks of $E$
and $F$.  So
with respect to a coordinate patch on $M$ and bases in $E,F$, an element of
$S^{m}(M;E,F)$ is just a matrix $a_{ij}$ of symbols of the
$S^{m}(B\x\R^{N};B\x \C^{k})$-kind earlier. A pseudodifferential operator
\[       D:C^{\infty}(M,E)\to C^{\infty}(M,F)            \]
is a map such that for every coordinate neighborhood $U$ of
$M$ trivializing $E,F$, the map that
sends $f\in C_{c}^{\infty}(U,E)$ to the restriction of $D\psi$ to $U$ is a
pseudodifferential operator of the form  (\ref{eq:ayD}),
or equivalently (\cite[p.84]{Hor}), that for any
$\phi,\psi\in C_{c}^{\infty}(U)$, the operator $\phi D\psi$ is a
pseudodifferential operator in the sense of (\ref{eq:ayD}).

The space of such pseudodifferential operators
$D$ is denoted by $\Psi^{m}(M;E,F)$ and is a Fr\'{e}chet space
under the seminorms $D\to \norm{(\phi D\psi)}_{m,U,\al,\bt,K}$.

The principle symbol $\si(D)$
of $D$, defined locally in (\ref{eq:spa}), is independent
of coordinate choices, and is a section of the bundle
$S^{m}(T_{0}^{*}M;Hom(\pi^{*}E,\pi^{*}F))$,
where $T_{0}^{*}M$ is $T^{*}M$ with the zero section removed and
$\pi$ is the canonical map from $T_{0}^{*}M$ to $M$.
An explicit formula for $\si(D)$ is:
\begin{equation} \label{eq:siD}
\si(D)(x,d\phi(x))f(x)=
\lim_{k\to \infty}k^{-m}e^{-ik\phi(x)}D(e^{ik\phi}f)(x)
\end{equation}
where $\phi\in C^{\infty}(M)$ and $f\in C^{\infty}(M,E)$ (e.g.
\cite[p.509]{AS1}, \cite[p.298]{CoMos}).

It is obvious from (\ref{eq:spa}) that
the principle symbol is homogeneous of degree $m$, i.e.
for $\xi\neq 0$ and $t>0$, we have
\[       a_{m}(x,t\xi)=t^{m}a(x,\xi).   \]
This homogeneity ennables one to regard the principle symbol as a section of
the sphere bundle $Hom(\pi_{S}^{*}E,\pi_{S}^{*}F)$ over $S(M)$. Here, $S(M)$
is the unit sphere bundle of $TM$ defined by some smooth Riemannian metric on
$M$, $\pi_{S}:S(M)\to M$ is the canonical projection, and
$\pi_{S}^{*}E$, $\pi_{S}^{*}F$ are the associated
pull-back bundles. Alternatively
expressed, $\si(D)$ is an element of $\text{Symb}^{m}(M;E,F)$,
the space of smooth
homomorphisms from $\pi_{S}^{*}E$ to $\pi_{S}^{*}F$.

For convenience, from now on, we will only consider the case where $E=F$ and
where the order of the pseudodifferential operators involved is $O$. For the
first of these, a well-known argument ennables us to replace $D$ by $D\oplus
D^{*}$ on $E\oplus F$ (using hermitian metrics on $E$ and $F$). So we can take
$E=F$.

One can reduce to the case $m=0$ by giving $E$ a hermitian metric and a
connection.  Then one just replaces $D$ by
by $(I+(D')^{*}D')^{-m/2}D$ where $D'$
is the covariant derivative of the connection (\cite[p.511]{AS1}).
We will write $\Psi(M;E)$ in place of $\Psi^{0}(M;E,E)$ and
$\text{Symb}(M;E)$ for $\text{Symb}^{0}(M;E,E)$.
An element $D\in \Psi(M;E)$
is called {\em elliptic} if the section $\si(D)$ is invertible.

Let $E$ be given a hermitian metric (so that $E$ becomes a finite-dimensional
Hilbert bundle).
Every $D\in \Psi(M;E)$ extends to a bounded linear operator on
$L^{2}(M,E,\mu)$ for any smooth positive measure $\mu$ on $M$ (e.g.
\cite[Theorem 6.5]{Shubin}). The closure of $\Psi(M;E)$ in $B(L^{2}(M,E,\mu))$
will be denoted by $\ov{\Psi}(M;E)$. For $D\in \Psi(M;E)$, let $\tilde{D}\in
\ov{\Psi}(M;E)$ be the extension of $D$ by continuity to $L^{2}(M,E,\mu)$.
The following result is due to Atiyah and Singer
(\cite[p.124]{AS4},\cite[p.512]{AS1}).

\begin{proposition}    \label{prop:contpsdo}
The map $D\to \tilde{D}$ is continuous. Further the map $\tilde{D}\to \si(D)$
is continuous, where the norm on $\text{Symb}(M;E)$ is the $\sup$-norm
over the unit sphere bundle $SM$ (treating $\pi_{S}^{*}E$ as a Hilbert bundle
over $SM$).
\end{proposition}

Let $(X,p)$ be (as earlier) a $G$-manifold which is a smooth fiber bundle over
$G^{0}$ with fiber $Z$.
Let $(\tE,q)$ be a {\em smooth complex vector bundle over
$X$} in the sense of Atiyah and Singer (\cite[p.123]{AS4}). So $\tE$ is a vector
bundle over $X$ such that $(\tE,p\circ q)$ is a fiber bundle over $G^{0}$ with
a fixed smooth complex vector bundle $E$ over $Z$ as fiber and structure
group $\text{Diff}(Z,E)$ (the topological group of diffeomorphisms of $\tE$
that map fibers to fibers linearly). Locally on $G^{0}$, the restrictions
$X\mid_{U}\sim U\x Z$ and $\tE\mid_{U}\sim U\x E$.

A natural example of a smooth vector bundle over $X$ is $TX$.

For $u\in G^{0}, x\in X$, let $\tE^{(u)}$ be the restriction of $\tE$ to
$X^{u}$ and $\tE^{x}$ be the $x$-fiber of $\tE$.
Note that $\tE^{(u)}$ is a vector bundle over $X^{u}$, while $\tE^{x}$ is a
vector space.
Let $C_{c}^{\infty}(X,\tE)$ be the space of smooth sections of $\tE$ with
compact supports. For $f\in C_{c}(X,\tE)$, let $f^{u}\in
C_{c}^{\infty}(X^{u},\tE^{(u)})$ be the restriction of $f$ to $X^{u}$.

Equivariance for the smooth fiber bundle $\tE$
is defined in a way similar to
that for group actions (e.g. \cite{Segal1}). Precisely, we require that
$\tE$ be
a $G$-manifold with $q:\tE\to X$ a $G$-map. In addition,
each $\ell_{g}:\tE^{(s(g))}\to \tE^{(r(g))}$, is a vector bundle isomorphism.
If $\tE$ is equivariant, then $\tE$ is a proper $G$-space
if $X$ is proper.

The fibered product $X*X$ is a fiber bundle over $G^{0}$ with fiber
$Z\x Z$ and structure group $\text{Diff}(Z\x Z)$.  In addition,
the external tensor product $\tE\boxtimes \tE$ is a
smooth vector bundle over $X*X$ with the smooth vector bundle $E\boxtimes E$ as
fiber.

\begin{proposition}   \label{prop:hermG}
There exists a hermitian metric $\lan ,\ran$ on $\tE$ which is $G$-isometric.
\end{proposition}
\begin{proof}
(cf. \cite[pp.40-41]{Phillips} for the group case.)
Let $\lan ,\ran'$ be a hermitian metric on $\tE$ and $c$ be as in
Proposition~\ref{prop:propc}. For $\xi,\eta\in \tE^{x}$,
define
\[  \lan\xi,\eta\ran^{x}=
\int c(g^{-1}x)\lan g^{-1}\xi,g^{-1}\eta\ran'\,d\la^{p(x)}(g).    \]
Applying the scalar version of
Lemma~\ref{lemma:cont} (to deal with continuity)
with $\tE\boxtimes \tE$ in place of $X$ and with
$f'(g,\xi\otimes \eta)=c(g^{-1}x)\lan g^{-1}\xi,g^{-1}\eta\ran'$,
we obtain a hermitian metric $x\to \lan ,\ran^{x}$ for $\tE$.  Arguing as in
the proof of Proposition~\ref{prop:etainvar} gives that $\lan ,\ran$ is
isometric.
\end{proof}

For the rest of the paper, $\tE$ (as well as $TX$) will be assumed to have
$G$-isometric hermitian metrics. The smooth version of the following definition
is given in \cite{LN,NWX}.

\begin{definition}      \label{def:psdo}
A {\em pseudodifferential family} on $X$
is a set $D=\{D^{u}\}$ $(u\in G^{0})$ where:
\be
\item each $D^{u}: C^{\infty}(X^{u};\tE^{(u)})\to
C^{\infty}(X^{u};\tE^{(u)})$ belongs to $\Psi(X^{u};\tE^{(u)})$;
\item given a basic open subset $V\sim p(V)\x Y$ of $X$, trivializing
$\tE$ (so that $\tE\mid_{V}\sim p(V)\x Y\x \C^{k}$),
there exists a continuous function
$a:p(V)\to S(T^{*}Y;Y\x M_{k})$ such that for each $u\in p(V)$, and
identifying $V\cap X^{u}$ with $Y$, we have
\begin{equation}
D_{u}\mid_{V\cap X^{u}}=a^{u}(x,D^{x})    \label{eq:auyD}
\end{equation}
where $a^{u}(x,\xi)=a(u)(x,\xi)$.
\ee
\end{definition}

Let $\Psi(X;\tE)$ be the set of pseudodifferential families above.
Definition~\ref{def:psdo} is equivalent to the corresponding definition
of a pseudodifferential family given in \cite{AS4}.  There, Atiyah and Singer
define $H$ to be the closed subgroup of $\text{Diff}(Z,E)\x \text{Diff}(Z,E)$
consisting of pairs $(\Psi,\Phi)$ where $\Psi,\Phi$ determine the same
element of $\text{Diff}(Z)$.  They show that $\Psi(X,\tE)$ is a fiber bundle
over $G^{0}$ with fiber $\Psi(Z;E)$ and structure group $H$.  Then $D$ is a
pseudodifferential family if and only if the map $u\to D^{u}$ is a continuous
section of $\Psi(X,\tE)$.

\begin{definition}
(cf. \cite{Cointeg})
A function $f\in C(X,\tE)$ is said to belong to \\
$C^{0,\infty}(X,\tE)$
if in local terms,
the map $u\to f^{u}$ from $G^{0}$ into $C^{\infty}(Z,E)$
is continuous (where $C^{\infty}(Z,E)$ has its standard topology
of uniform convergence on compacta for all derivatives).  The space of
functions $f$ in $C^{0,\infty}(X,\tE)$ that have compact support is denoted by
$C_{c}^{0,\infty}(X,\tE)$.
\end{definition}

It is obvious that $C_{c}^{0,\infty}(X,\tE)$ is a complex vector space which
contains $C_{c}^{\infty}(X,\tE)$.  The space $C_{c}^{0,\infty}(X,\tE)$ plays a
similar role for the family $D=\{D_{u}\}$ as does $C_{c}^{\infty}(M,E)$ does
for pseudodifferential operators on manifolds.

\begin{proposition}       \label{prop:propD}
Let $D$ be a pseudodifferential family on $X$.  Then
$D:C^{0,\infty}(X,\tE)\to C^{0,\infty}(X,\tE)$,
where for $x\in X^{u}$, we define
\[                   Df(x)= D^{u}(f^{u})(x).                          \]
\end{proposition}
\begin{proof}
Using a partition of unity argument (cf. \cite[p.23]{Eg}), we can suppose that
there is a basic open set $V\sim p(V)\x Y$ in $X$
trivializing $\tE$ and functions
$\phi,\psi\in C_{c}^{\infty}(V)$ such that $D=\phi D\psi$. Since $u\to (\psi
f)^{u}$ is continuous from $G^{0}$ into $C_{c}^{\infty}(Y,\C^{k})$ and $D$ is a
continuous family, it follows by the joint continuity of pseudodifferential
operators on $C_{c}^{\infty}$ functions (\cite[Theorem 18.1.6]{Hor}) that
$Df\in C_{c}^{0,\infty}(X,\tE)$.
\end{proof}

\begin{definition}
A {\em parametrix} for $D\in \Psi(X;\tE)$ is a pseudodifferential family
$P=\{P^{u}\}$ such that for each $u$, $P^{u}$ is a parametrix for $D^{u}$, i.e.
each of $D^{u}P^{u}-I, P^{u}D^{u}-I$ is a ``smoothing'' operator, an element
$T\in \Psi(X^{u};\tE^{u})$ for which $\si(T)=0$.
\end{definition}

\begin{proposition} \label{prop:symb}
Let $D$ be an elliptic,
pseudodifferential family on $X$. Then there exists a
parametrix $P$ for $D$.
\end{proposition}
\begin{proof}
Note that $SX$ is a manifold over
$G^{0}$ with fiber $SZ$, and that $End(\pi_{S}^{*}(\tE))$ is a smooth vector
bundle over $SX$.
The vector bundle $End(\pi_{S}^{*}(\tE))$
restricts on each $X^{u}$ to $\pi_{S,u}^{*}(X^{u};End(\tE^{(u)}))$, where
$\pi_{S,u}:SX^{u}\to X^{u}$ is the canonical projection.
The ``principle
symbol'' $\si(D)$ of $D$, where $\si(D)(u)=\si(D^{u})$, is a section of the
vector bundle $End(\pi_{S}^{*}(\tE))$. In fact, $\si(D)$ is continuous -- this
follows from the continuity of the maps $u\to D^{u}\to \si(D^{u})$.

By ellipticity, the symbol $\si(D)^{-1}$ exists. In addition,
straight-forward computations show that, locally,
the matrices $b(u,x,\xi)$ belong
to $S^{0}$, and that $u\to b^{u}$ is continuous.  Following the argument of
\cite[Ch. IV,Theorem 3.15]{Wells} (which deals with the case of a single
pseudodifferential operator. i.e. where $G^{0}$ is a singleton), using basic
open sets for an open cover, one constructs a pseudodifferential family
$P$ on $X$ whose symbol is $\si(D)^{-1}$.  Then
(\cite[Ch. IV, Theorem 4.4]{Wells}) each $P^{u}$ is a parametrix for $D^{u}$.
So $P$ is a parametrix for $D$.  (It can be shown that a product of two
pseudodifferential families is a pseudodifferential family
(cf. \cite[Theorem 3.16]{Wells}) but we won't need this.)
\end{proof}

\section{Construction of the analytic index}
We now consider invariant pseudodifferential families on the $G$-manifold $X$.
To this end, there is a section action $g\to L_{g}$ of $G$ on
$C^{\infty}_{c}(X,\tE)$, where
$L_{g}:C^{\infty}(X^{s(g)},\tE^{(s(g))})\to
C^{\infty}(X^{r(g)},\tE^{(r(g))})$ is the
diffeomorphism given by:
\begin{equation}  \label{eq:Lg}
L_{g}f(x)=g[f(g^{-1}x)]\hspace{.1in} (x\in X^{r(g)}).
\end{equation}
There is a natural (algebraic) action of the groupoid $G$ on $\Psi(X,\tE)$
given by: $g\to L_{g}D^{s(g)}L_{g^{-1}}$.
Note that $L_{g}D^{s(g)}L_{g^{-1}}$
belongs to $\Psi(X^{r(g)};E^{(r(g))})$ by \cite{Hor2}.

\newcommand{\bga}{\textbf{B}(\gah)}
\newcommand{\kga}{\textbf{K}(\gah)}
\newcommand{\bmfH}{\textbf{B}(\mfH)}
\newcommand{\kmfH}{\textbf{K}(\mfH)}
\newcommand{\bgga}{\textbf{B}_{G}(\mfH)}
\newcommand{\kgga}{\textbf{K}_{G}(\mfH)}
\newcommand{\gach}{\Ga_{c}(\mfH)}
\newcommand{\gah}{\Ga(\mfH)}

\begin{definition}    \label{def:invpsdo}
The pseudodifferential family $D$ is called {\em invariant} if
\begin{equation}  \label{eq:rgd}
L_{g}D^{s(g)}L_{g^{-1}}=D^{r(g)}.
\end{equation}
\end{definition}

The above notion is well-known in the literature (e.g.
\cite{Cointeg,Connesbook,MontP,NWX}). We now discuss the
symbol $\si(D)$ of an invariant pseudodifferential family $D$.

Since $D$ is invariant then its symbol $\si(D)$ is also
invariant under the natural action of $G$ on the sections of
$\pi_{S}^{*}End(\tE))$.
The action of $G$ on the sections of this bundle is given by:
$g\to g\al
g^{-1}$. The map $D\to \si(D)$ is equivariant. One easy way to prove this is
just to calculate the symbol of $g\si(D^{s(g)})g^{-1}$ using the explicit
formula (\ref{eq:siD}) and the fact that $\ell_{g}$ is a diffeomorphism from
$X^{s(g)}$ onto $X^{r(g)}$. The details are left to the reader.

Suppose now that $D$ is, in addition, elliptic, and let $P$ be a parametrix
for $D$ (Proposition~\ref{prop:symb}).
There is actually an invariant parametrix $P_{1}$ for $D$.  This can be
proved by taking
\[  P_{1}^{u}f(x)=\int_{G^{u}}c(g^{-1}x)
(L_{g}P^{s(g)}L_{g^{-1}})f(x)\,d\la^{u}(g)   \]
where $p(x)=u$,
$f\in C^{\infty}(X^{u},\tE^{(u)})$, and $c$ is smooth and as in
Proposition~\ref{prop:propc}.  In this connection, cf.
\cite{CoMos,Kasp1,Phillips}. However, rather than giving the details of the
proof of the existence of an invariant parametrix $P_{1}$, it is more
convenient to deal with the corresponding question at the Fredholm level
later (Proposition~\ref{prop:parinv}), where the corresponding proof in the
group case by Phillips adapts easily.

For the rest of this paper, $D$ is an invariant, elliptic pseudodifferential
family on $X$ as above. The analytic index of $D$ will be constructed by
adapting the approach of N. C. Phillips in \cite{Phillips} to equivariant
K-theory for proper actions. (This in turn was motivated by the work of Segal
(\cite{Segal1,Segal2}.)

We first construct a Hilbert bundle $(\mfH,\tp)$ over $G^{0}$. By definition
(\cite[p.7]{Phillips}), a Hilbert bundle is a locally trivial fiber bundle with
a Hilbert space $H$ as fiber and with structure group $U(H)$, the unitary
group of $H$ with the strong operator topology. The fiber over $u\in G^{0}$ of
this bundle is $L^{2}(X^{u},\tE^{(u)},\mu^{u})$.
The map $\tp$ just takes any
$f\in L^{2}(X^{u},\tE^{(u)},\mu^{u})$ to $u$.

\begin{proposition}          \label{prop:hbundle}
Let $\mu$ be a smooth measure on $Z$ and give $E$ a hermitian metric
$\lan ,\ran$. Then, in a natural way, the bundle
$\mfH$ is a Hilbert bundle with fiber $L^{2}(Z,E,\mu)$, and it admits a
canonical continuous unitary action of the Lie groupoid $G$.
\end{proposition}
\begin{proof}
We will establish the Hilbert
bundle structure of $\mfH$ by constructing a
cocycle $\{g_{\al\bt}\}$ for that bundle (e.g. \cite[p.48]{BottTu}).
Let $\{(U_{\al},\phi_{\al})\}$ be an open cover of $G^{0}$ which trivializes the fiber
bundle $(\tE,\pi\circ p)$.  So each $\phi_{\al}$ is a
fiber preserving homeomorphism from
$\tE\mid_{U_{\al}}$ onto $U_{\al}\x E$ which is a vector bundle isomorphism
on fibers.  For each $u$, let
$\tau_{\al}^{u}:X^{u}\to Z$ be the diffeomorphism
determined by $\phi_{\al}$.
Let $x\in X$, and set $p(x)=u$,
$\tau_{\al}^{u}(x)=z$. Let $\phi_{\al}^{u}$ be the
restriction of $\phi_{\al}$ to $\tE^{(u)}$. Then there exists a positive
definite, invertible
element $A_{\al}^{x}\in End(E^{z})$
such that for all $\xi\in \tE^{x}$, we have
$\lan \xi,\xi\ran^{x}=
\lan A_{\al}^{x}\phi_{\al}(\xi),\phi_{\al}(\xi)\ran^{z}$,
and the map $x\to A_{\al}^{x}$ is continuous. Define
$\chi_{\al}:\mfH\mid_{U_{\al}}\to U_{\al}\x L^{2}(Z,E,\mu)$ by:
\[      \chi_{\al}f^{u}(z)=
r_{\al}^{u}(z)^{1/2}(A_{\al}^{x})^{1/2}(f^{u}\circ (\phi_{\al}^{u})^{-1})   \]
where $r_{\al}^{u}=d((\tau_{\al}^{u})^{*}\mu^{u})/d\mu$.
It is left to the reader to check that the map
$\chi_{\al}\chi_{\bt}^{-1}:U_{\al}\cap U_{\bt}\to
B(L^{2}(Z,E,\mu))$ is a cocycle with values in $U(L^{2}(Z,E,\mu))$.  Then
$\mfH$ is the Hilbert bundle constructed in the standard way with
transition functions $g_{\al\bt}=\chi_{\al}\chi_{\bt}^{-1}$.  To check that
$g_{\al\bt}$ is continuous into $U(L^{2}(Z,E,\mu))$ with the strong operator
(=weak operator) topology, one uses elementary measure theory.

Turning to the groupoid action on $\mfH$, for each $g\in G$,
by the invariance
of the $\mu^{u}$'s and the isometric action of $G$ on $\tE$, the map $L_{g}$
extends from $C(X^{s(g)},E^{(s(g))})$ to give a unitary element of
$B(L^{2}(X^{s(g)},\tE^{(s(g))},\mu^{s(g)}),
L^{2}(X^{r(g)},\tE^{(r(g))},\mu^{r(g)}))$.
Clear\-ly, algebraically, $\mfH$ is a $G$-space with unitary action $g\to L_{g}$
where $L_{g}$ has the same formula as in (\ref{eq:Lg}).
It remains to be shown that
the product map $(g,f)\to L_{g}f$ from $G*\mfH$ into $\mfH$ is continuous.
Suppose then that $g_{n}\to g$ in $G$, $f_{n}\to f$ in $\mfH$ with
$s(g_{n})=\tp(f_{n})$ and $s(g)=\tp(f)$.  Translating this into local terms, we can
suppose that $U,V$ are open subsets
of $G^{0}$, such that $s(g_{n}), s(g)\in U, r(g_{n}), r(g)\in V$ and $Z,\tE$
are trivial over $U$ and $V$. In
addition, we can suppose that $f_{n}, f\in L^{2}(Z,E,\mu)$, and, regarding the
$L_{g_{n}}, L_{g}$ as unitary on $L^{2}(Z,E,\mu)$, we have to show that
$\norm{L_{g_{n}}f_{n}-L_{g}f}_{2}\to 0$. Given $\eps>0$, there exists $F\in
C_{c}(Z,E)$ such that $\norm{F-f}_{2}<\eps$.  By the continuity of the action
of $G$ on $X$
we have $\norm{L_{g_{n}}F-L_{g}F}_{\infty}\to 0$, and so
$\norm{L_{g_{n}}F-L_{g}F}_{2}\to 0$. An elementary triangular inequality
argument then shows that $\norm{L_{g_{n}}f_{n}-L_{g}f}_{2}<\eps$ eventually, so
that $\norm{L_{g_{n}}f_{n}-L_{g}f}_{2}\to 0$ as required.
\end{proof}

We now recall some facts about morphisms on Hilbert bundles 
(\cite[Chapter 1]{Phillips}). A {\em morphism} on $\mfH$ is a continuous map
$T:\mfH\to \mfH$ which restricts to a linear map $T^{u}$ on each fiber
$\mfH^{u}$ and is such that the adjoint map $T^{*}$ on $\mfH$, where
$(T^{*})^{u}=(T^{u})^{*}$, is also continuous. The morphism $T$ is called
{\em equivariant} if for each $g\in G$, we have
\begin{equation}
L_{g}T^{s(g)}L_{g^{-1}}=T_{r(g)}.        \label{eq:LgTrg}
\end{equation}
In particular, if $T$ is equivariant, then
\begin{equation}
\norm{T^{s(g)}}=\norm{T^{r(g)}}  \label{eq:tsgrg}
\end{equation}
for all $g$.

The morphism $T$ is called {\em
bounded} if
\begin{equation}
\norm{T}=\sup_{u\in G^{0}}\norm{T^{u}}<\infty.   \label{eq:normT}
\end{equation}
The set of all bounded morphisms on $\mfH$ is a unital $\cs$ $\bmfH$
in the obvious way under the norm of
(\ref{eq:normT}).

It follows by the $G$-compactness of $G^{0}$, the local boundedness of
morphisms (\cite[Cor. 1.8]{Phillips})
and (\ref{eq:tsgrg}) that every
equivariant morphism is automatically bounded.  The set of equivariant
morphisms of $\mfH$ is a unital $C^{*}$-subalgebra of $\bmfH$, and will be
denoted by $\bgga$.

An element $T\in \bmfH$ is called {\em compact} if for every compact
subset $A$ of $G^{0}$, the set $\{T^{u}(\xi^{u}): \xi^{u}\in \mfH^{u},
\norm{\xi^{u}}\leq 1, u\in A\}$ has compact closure in $\mfH$.      
The set of bounded compact morphisms is a closed ideal of $\bmfH$
and is denoted by $\kmfH$ (cf. \cite[Lemma 1.12]{Phillips}).  Similarly,
the set $\kgga$ of equivariant compact morphisms is a closed ideal of
$\bgga$.

Next we introduce the equivariant compact morphisms which correspond to the
``rank 1'' operators on  a Hilbert module.
Let $e_{1}, e_{2}\in C_{c}(X,\tE)$ and define
a section $h_{e_{1},e_{2}}$ of $\tE\boxtimes \tE$ by:
\begin{equation}  \label{eq:he1e2}
h_{e_{1},e_{2}}(x,y)=
\int L_{g}e_{1}(x)\boxtimes L_{g}e_{2}(y)\,d\la^{p(x)}(g).
\end{equation}
Note that the section $h_{e_{1},e_{2}}$ in (\ref{eq:he1e2}) has compact
support and is continuous by the section version of
Proposition~\ref{prop:etainvar}.

Now define the operator $T(e_{1},e_{2})^{u}$ on
$\mfH^{u}$ by:
\begin{equation} \label{eq:te1e2}
T(e_{1},e_{2})^{u}(\xi)(x)=\lan \xi, h_{e_{1},e_{2}}(x,.)\ran
\end{equation}
where the inner product on the right-hand side of (\ref{eq:te1e2}) is
calculated in $\mfH^{u}=L^{2}(X^{u},\tE^{(u)},\mu^{u})$.
Then $T(e_{1},e_{2})^{u}$ is a compact operator since it is a kernel operator
whose kernel is continuous with compact support.  The equation (\ref{eq:te1e2})
can be usefully written:
\begin{equation}  \label{eq:te1e2'}
T(e_{1},e_{2})^{u}(\xi)=\int L_{g}e_{1}\ov{\lan L_{g}e_{2},\xi\ran}\,d\la^{u}(g).
\end{equation}

\begin{proposition}  \label{prop:span}
The map $T(e_{1},e_{2})\in \kgga$, and the span of morphisms of the form
$T(e_{1},e_{2})$ is dense in $\kgga$.
\end{proposition}
\begin{proof}
We first show that $T(e_{1},e_{2})$ is well-defined and is
continuous on $\mfH$.  We just need to
show this on some trivialization $U\x L^{2}(Z,E,\mu)$ of $\mfH$.
Applying Lemma~\ref{lemma:cont} with $X*X$ in place of $X$,
$\mathcal{E}=\tE\boxtimes \tE$ and
\[  f'(g,x,y)=L_{g}e_{1}(x)\boxtimes L_{g}e_{2}(y)   \]
gives that $h_{e_{1},e_{2}}$ is continuous on $X*X$.
It follows that $T(e_{1},e_{2})$ is continuous
on $\mfH$ and that for any compact subset $A$ of $U$
(and hence of $G^{0}$) the set
\[  \{T(e_{1},e_{2})^{u}(\xi^{u}): \xi^{u}\in \mfH^{u},
\norm{\xi^{u}}\leq 1, u\in A\} \]
has compact closure in $\mfH$. Next, $T(e_{1},e_{2})$ is adjointable with
adjoint $T(e_{2},e_{1})$. So $T(e_{1},e_{2})\in \kmfH$.  We now show
that $T=T(e_{1},e_{2})$ is invariant. This follows from the argument below using
(\ref{eq:te1e2'}) and (\ref{eq:LgTrg}):
\beqns
T^{r(g_{0})}L_{g_{0}}\xi &=&
\int L_{g}e_{1}\ov{\lan L_{g}e_{2},L_{g_{0}}\xi\ran}\,d\la^{r(g_{0})}(g)  \\
&=& \int L_{g_{0}}[L_{g_{0}^{-1}g}e_{1}]
\ov{\lan L_{g_{0}^{-1}g}e_{2},\xi\ran}\,d\la^{r(g_{0})}(g)  \\
&=& \int L_{g_{0}}[L_{h}e_{1}]
\ov{\lan L_{h}e_{2},\xi\ran}\,d\la^{s(g_{0})}(h)  \\
&=& L_{g_{0}}T^{s(g_{0}}\xi.
\eeqns

For the last part of the proposition, we have to show that given
$R\in \kgga$ and $\eps>0$, then
there exist $\xi_{i},\eta_{i}$ $(1\leq i\leq n)$ in
$C_{c}(X,\tE)$ such that
\begin{equation}  \label{eq:rTxi}
\norm{R^{u} - \sum_{i=1}^{n} T(\xi_{i},\eta_{i})^{u}}\leq\eps
\end{equation}
for all $u\in G^{0}$.  To this end, we adapt the corresponding argument of
Phillips in the locally compact group case (\cite[pp.92-93]{Phillips}).
Restricting $R$ to  trivializing subsets of $G^{0}$, there is then
an open cover $\{U_{u}\}$ of $G^{0}$ such that for each $u$,
there exist $\xi_{i,u},\eta_{i,u}\in C_{c}(X,\tE)$ such that
\[  \norm{R^{v}-\sum_{i}(\xi_{i,u}\otimes \eta_{i,u})^{v}}<\eps   \]
for all $v\in U_{u}$.  Here
$(\xi_{i,u}\otimes \eta_{i,u})^{v}(w)=
(\xi_{i,u})^{v}\ov{\lan (\eta_{i,u})^{v},w\ran}$ for $w\in
L^{2}(X^{v},\tE^{(v)},\mu^{v})$.  Using
Proposition~\ref{prop:aver} with $G^{0}$ in place of $X$, there exists
a $G$-partition of unity
$\{f_{\ga}\}$ subordinate to the cover $\{U_{u}\}$ of $G^{0}$.  By
considering terms of the form
$f_{\ga}^{1/2}\xi_{i,u}\otimes f_{\ga}^{1/2}\eta_{i,u}$ and using the
$G$-compactness of $G^{0}$ and the invariance of the morphisms involved,
there exist
$\xi_{j},\eta_{j}\in C_{c}(X,\tE)$
($j$ in some finite index set $J$) such that for all $u\in G^{0}$,
\[  \norm{h(u)R^{u}-\sum_{j}(\xi_{j}\otimes \eta_{j})^{u}}\leq \eps h(u) \]
where $h=\sum_{\ga} f_{\ga}$. Then for any $u\in G^{0}$, we have
\beqns
\lefteqn{\norm{R^{u}-\int\sum_{j}L_{g}(\xi_{j}\otimes
\eta_{j})L_{g^{-1}}\,d\la^{u}(g)} } \\
& = & \norm{\int[h(g^{-1}u)R^{u}-\sum_{j}L_{g}(\xi_{j}\otimes\eta_{j})
L_{g^{-1}}]\,d\la^{u}(g)}\\
&=&\norm{\int L_{g}[h(g^{-1}u)R^{g^{-1}u}-
\sum_{j}(\xi_{j}\otimes\eta_{j})^{g^{-1}u}]L_{g^{-1}}\,d\la^{u}(g)} \\
&\leq & \int\norm{h(g^{-1}u)R^{g^{-1}u}-
\sum_{j}(\xi_{j}\otimes\eta_{j})^{g^{-1}u}}\,d\la^{u}(g) \\
&\leq &\int\epsilon h(g^{-1}u)\,d\la^{u}(g)= \eps.
\eeqns
Since $L_{g}(\xi_{j}\otimes\eta_{j})L_{g^{-1}}
=L_{g}\xi_{j}\otimes L_{g}\eta_{j}$,
we obtain (\ref{eq:rTxi}).
\end{proof}

An element $T\in \bmfH$ is called {\em Fredholm} if there exists
$S\in \bmfH$ such that both $ST-I, TS-I\in \kmfH$.

\begin{proposition}\label{prop:parinv} 
\be
\item Suppose that $T\in \bgga$ is Fredholm. Then there exists
$S\in \bgga$ such that both $ST-I,TS-I\in \kgga$.
\item  The elliptic pseudodifferential family $D=\{D^{u}\}$ defines, by
extending each $D^{u}$ to $\tilde{D}^{u}\in B(\mfH^{u})$
(Proposition~\ref{prop:contpsdo}) an invariant Fredholm morphism on $\mfH$.
\ee
\end{proposition}
\begin{proof}
(i) The locally compact group version of this is given in \cite[Lemma
3.7]{Phillips}.
For each $u\in G^{0}$, define
$A^{u}\in B(\mfH^{u})$ by: for $\xi\in \mfH^{u}$, $x\in X^{u}$, set
\begin{equation}
A^{u}\xi(x)=\int_{G^{u}}
c(g^{-1}x)L_{g}S^{s(g)}L_{g^{-1}}f(x)\,d\la^{u}(g)  \label{eq:cgpar}
\end{equation}
where $c$ is as in Proposition~\ref{prop:propc}.
The argument of Phillips adapts directly to show that $A\in \bgga$ and
satisfies $TA-I,AT-I\in \kgga$.

(ii) By Proposition~\ref{prop:contpsdo}, in local terms, the map
$u\to \tilde{D}^{u}$, regarded as a map into $B(L^{2}(Z,E,\mu))$,
is strong operator -- even norm -- continuous,
and has an adjoint $u\to (\tilde{D}^{u})^{*}$.  So $D$, identified with
$\{\tilde{D}^{u}\}$, is a morphism.
The invariance of this morphism
follows from the invariance of each $D^{u}$ on the dense subspace
$C^{\infty}(X^{u},E^{(u)})$ of $L^{2}(X^{u},\tE^{(u)},\mu^{u})$.
\end{proof}

We now construct the Fredholm module which will give the index of the elliptic
pseudodifferential family $D$. The proof constructs a certain Kasparov $(\C,
C_{red}^{*}(G))$ module. The proof in the locally compact group case is
effectively given by Phillips (\cite[Ch. 6]{Phillips}) in his discussion of the
generalized Green-Rosenberg theorem. Actually, Phillips ({\em loc. cit.})
proves more than that, showing that his equivariant K-theory group
$K_{G}^{0}(X)$ is isomorphic as an abelian group to $K_{0}(C^{*}(G,X))$. We
will not consider the groupoid version of this in this paper but consider only
the corresponding Kasparov $(\C,C_{red}^{*}(G))$-module that gives the analytic
index.

Let $\Ga_{c}(\mfH)=C_{c}(G^{0},\mfH)$.  For $T\in \bgga$ and $f\in \gach$,
define a section $\Phi(T)(f)$ of $\mfH$ by:
\begin{equation}
\Phi(T)(f)(u)=T^{u}f\hspace{.1in}(=T^{u}(f(u))).           \label{eq:PhiT}
\end{equation}

\begin{proposition}    \label{prop:phiga}
Let $f\in \gach$.  Then the section $\Phi(T)(f)$ belongs to $\gach$.
\end{proposition}
\begin{proof}
Trivially, the support of $\Phi(T)$ is compact in $G^{0}$.  Let $u\in G^{0}$
and $u_{n}\to u$ in $G^{0}$.  Trivializing in a neighborhood of $u$ in
$G^{0}$, we can regard $f(u_{n}), f(u)\in L^{2}(Z,E,\mu)$
(cf. Proposition~\ref{prop:hbundle}), and by the continuity of $f$,
$\norm{f(u_{n}) - f(u)}_{2}\to 0$. Since $T$ is continuous on $\mfH$, we have
$T_{u_{n}}f\to T_{u}f$.  So $\Phi(T)(f)$ belongs to $\gach$.
\end{proof}

We will show that $\gach$ is a
pre-Hilbert module over the pre-$\cs$ $C_{c}(G)$, with $C_{c}(G)$-inner
product and module operations given
below, where $g\in G$, $x\in X$, $e,e_{1},e_{2}\in \gach$ and $f\in C_{c}(G)$:
\beqn
\lan e_{1},e_{2}\ran(g)&=&\int_{X^{r(g)}}
\ov{\lan e_{1}(x),L_{g}e_{2}(x)}\ran\,
d\mu^{r(g)}(x)    \label{eq:lanran}      \\
ef(x)&=&\int_{G^{p(x)}}L_{g}e(x)f(g^{-1})\,d\la^{p(x)}(g).   \label{eq:efx}
\eeqn
The preceding equality can alternatively be written as:
\begin{equation}
 (ef)(u)=\int_{G^{u}}(L_{g}e)f(g^{-1})\,d\la^{u}(g).  \label{eq:efx2}
\end{equation}

The following equality is useful:
\begin{equation}
\lan e_{1},L_{g}e_{2}\ran = \lan L_{g^{-1}}e_{1}, e_{2}\ran   \label{eq:ege}
\end{equation}
where we have omitted the superscripts $r(g), s(g)$ respectively on the two
preceding inner products.
This follows from  (\ref{eq:invarmu}) and the isometric action of $G$ on
$\tE$.

Note that (\ref{eq:lanran}) can then be rewritten succinctly as:
\begin{equation}
\lan e_{1},e_{2}\ran (g)=\ov{\lan e_{1},L_{g}e_{2}\ran}.      \label{eq:e1ge2}
\end{equation}

\begin{proposition}       \label{prop:preH}
The space $\gach$ is a 
pre-Hilbert module over
the pre-$C^{*}$-alg\-ebra $C_{c}(G)\subset C_{red}^{*}(G)$ with $C_{c}(G)$-inner product and module
action given by (\ref{eq:lanran}) and (\ref{eq:efx}) respectively.
\end{proposition}
\begin{proof}
We claim first that $\lan e_{1},e_{2}\ran\in C_{c}(G)$. That $\lan
e_{1},e_{2}\ran$ is continuous follows using (\ref{eq:e1ge2}), the continuity of
the sections $e_{1},e_{2}$ and the strong operator continuity of the groupoid
action $g\to L_{g}$ on $\mfH$.  Now let $C_{i}$ be the (compact) support of
the sections $e_{i}$. By the properness of the action of $G$ on $X$,
it follows that the set $\{(g,x): (g^{-1}x,x)\in C_{2}\x C_{1}\}$ is compact,
and the support of $\lan e_{1},e_{2}\ran$ is contained in the projection of
that set onto the first coordinate. A similar argument (using (\ref{eq:efx2}))
shows that $ef\in \gach$.

It remains to check (\cite[p.126]{Blackadar}) that (a) $\lan, \ran$ is
sesquilinear, and that for all $e,e_{1},e_{2}\in \gach$ and all
$f\in C_{c}(G)$, we have (b) $\lan e_{1},e_{2}f\ran=\lan e_{1},e_{2}\ran f$,
(c) $\lan e_{1},e_{2}\ran^{*}=\lan e_{2},e_{1}\ran$, and that (d) $\lan
e,e\ran\geq 0$, and is $0$ if and only if $e=0$. (a) is obvious.
We prove the others in turn.

For (b), using (\ref{eq:convo}):
\beqn
\lan e_{1},e_{2}f\ran(g)&=&
\int\ov{\lan e_{1}(x),L_{g}(e_{2}f)(x)\ran}\,d\mu^{r(g)}(x) \nonumber\\
&=& \int\ov{\lan e_{1}(x),g[e_{2}f(g^{-1}x)]\ran}\,d\mu^{r(g)}(x)\nonumber  \\
&=& \int\int \ov{\lan e_{1}(x),g[L_{k}(e_{2})(g^{-1}x)]\ran}f(k^{-1})\,
d\la^{s(g)}(k)d\mu^{r(g)}(x) \nonumber \\
&=& \int\int \ov{\lan e_{1}(x),L_{gk}e_{2})(x)\ran} f(k^{-1})\,
d\la^{s(g)}(k)d\mu^{r(g)}(x) \nonumber \\
&=&\int\int\ov{\lan e_{1}(x),L_{h}e_{2}(x)\ran} f(h^{-1}g)\,
d\la^{r(g)}(h)d\mu^{r(g)}(x)  \nonumber\\
&=& \int \lan e_{1},e_{2}\ran(h)f(h^{-1}g)\,d\la^{r(g)}(h) \\
&=& \lan e_{1},e_{2}\ran f(g).  \nonumber
\eeqn

For $(c)$, using (\ref{eq:ege}) and (\ref{eq:e1ge2}), we have
$\lan e_{2},e_{1}\ran^{*}(g)= \ov{\lan e_{2},L_{g^{-1}}e_{1}\ran}
=\ov{\lan L_{g}e_{2},e_{1}\ran}
=\lan e_{1},e_{2}\ran(g).$

For (d),
from the definition of $C_{red}^{*}(G)$ in \S 2, we just have to show that for
$u\in G^{0}$, we have
$\pi_{u}(\lan e,e\ran)\geq 0$ where $\pi_{u}$ was defined in
(\ref{eq:red}).
Taking $f=\lan e,e\ran$, we get, by \cite[(3.42)]{Paterson}, that
for $\xi\in C_{c}(G)$
\beqn
\lan\pi_{u}(f)\xi,\xi\ran &=&
\int\int f(gh)\xi(h^{-1})\ov{\xi(g)}\,d\la^{u}(h)\la_{u}(g) \nonumber \\
&=& \int\int
\ov{\lan e,L_{gh}e\ran}\xi(h^{-1})\ov{\xi(g)}\,d\la^{u}(h)d\la_{u}(g) \nonumber \\
&=& \int\int\ov{\lan L_{g^{-1}}e,L_{h^{-1}}e\ran}\xi(h)\ov{\xi(g)}\,
d\la_{u}(h)d\la_{u}(g) \nonumber \\
&=&\lan\int\xi(g)L_{g^{-1}}e\,d\la_{u}(g),
\int\xi(h)L_{h^{-1}}e\,d\la_{u}(h)\ran_{L^{2}(X^{u},\tE^{(u)},\mu^{u})}
\label{eq:piufxi} \\
&\geq & 0.
\eeqn

Since for $u\in G^{0}$, $\lan e,e\ran(u)=\int\norm{e(x)}^{2}\,d\la^{u}(g)$,
it follows that $\lan e,e\ran=0$ if and only if $e=0$.
\end{proof}

It is elementary that the $C_{c}(G)$-valued inner product $\lan ,\ran$
on $\gach$ extends to a $C_{red}^{*}(G)$-valued inner product, also denoted by
$\lan ,\ran$, on the completion $\gah$ of $\gach$ under the norm
$e\to \norm{\lan e,e\ran}^{1/2}$, and is a Hilbert $C_{red}^{*}(G)$-module.
The $\css$ of bounded and compact module maps on $\gah$ will be denoted by
$\bga$ and $\kga$ respectively.  In the following theorem, the map $\Phi$
actually determines an isomorphism from $\bgga$ onto $\bga$ but we
don't include the proof since it is not needed for our purposes.

\begin{theorem}  \label{th:morK}
The map $\Phi$ of Proposition~\ref{prop:phiga}
determines a $C^{*}$-isomorphism from $\bgga$
into $\bga$ that takes $\kgga$ onto $\kga$.
\end{theorem}
\begin{proof}
Trivially, $\Phi(T):\gach\to \gach$ is linear. We will show first that is
continuous with norm $\leq \norm{T}$. In particular, it extends by continuity
to a bounded linear map on $\gah$.

Let $e\in \gach$ and $T\in \bgga$. It is sufficient to show that
\begin{equation}   \label{eq:Tcont}
\norm{\lan \Phi(T)e,\Phi(T)e\ran}\leq \norm{T}^{2}\norm{\lan e,e\ran}.
\end{equation}
This will follow if we can show that for
$u,\xi$ as in the proof of (d) of Proposition~\ref{prop:preH},
\begin{equation}   \label{eq:piulan}
\lan\pi_{u}(\lan \Phi(T)e,\Phi(T)e\ran)\xi,\xi\ran
\leq \norm{T}^{2}\lan\pi_{u}(\lan e,e\ran)\xi,\xi\ran.
\end{equation}

Using (\ref{eq:piufxi}), the invariance of $T$ and 
the fact that (in the proof) $s(g)=u$:
\beqns
\lan\pi_{u}(\lan \Phi(T)e,\Phi(T)e\ran)\xi,\xi\ran
&=& \norm{\int\xi(g)L_{g^{-1}}T^{r(g)}e\,d\la_{u}(g)}_{2}^{2}\\
&=& \norm{\int\xi(g)T^{u}L_{g^{-1}}e\,d\la_{u}(g)}_{2}^{2}\\
&\leq &\norm{T_{u}}^{2}\norm{\int\xi(g)L_{g^{-1}}e\,d\la_{u}(g)}^{2}
\eeqns
and (\ref{eq:piulan}) follows.

So $\Phi(T)$ is a bounded linear operator on $\gah$ and $\norm{\Phi(T)}\leq
\norm{T}$. To obtain $\Phi(T)\in \bga$, we show that
$\Phi(T)$ is an
adjointable module map on $\gah$ with adjoint $\Phi(T^{*})$.

For $e_{1},e_{2}\in \gah$ and $g\in G$, we have using (\ref{eq:e1ge2}) and
the invariance of $T^{*}$,
\beqns
\lan\Phi(T)e_{1},e_{2}\ran(g) &=& \ov{\lan\Phi(T)e_{1},L_{g}e_{2}\ran}  \\
&=& \ov{\lan T^{r(g)}e_{1},L_{g}e_{2}\ran}            \\
&=& \ov{\lan e_{1},(T^{r(g)})^{*}L_{g}e_{2}\ran}  \\
&=& \ov{\lan e_{1},L_{g}(T^{s(g)})^{*}e_{2}\ran}  \\
&=& \lan e_{1},\Phi(T^{*})e_{2}\ran.  \\
\eeqns
So $\Phi(T^{*})$ is an adjoint for $\Phi(T)$.

Next we have to show that $\Phi(T)$ is a module map.  To this end, for
$e\in \gach,f\in C_{c}(G)$, we have by (\ref{eq:efx2}),
\beqns
[\Phi(T)(ef)](u)&=& T^{u}(ef)(u)       \\
&=& T^{u}\int L_{g}(e)f(g^{-1})\,d\la^{u}(g)       \\
&=& \int T^{r(g)}L_{g}(e)f(g^{-1})\,d\la^{u}(g)       \\
&=& \int L_{g}T^{s(g)}(e)f(g^{-1})\,d\la^{u}(g)       \\
&=& (Te)f(u).
\eeqns
So $\Phi(T)$ is a module map.  From (\ref{eq:PhiT}), $\Phi$ is a
$^{*}$-homomorphism that is one-to-one.
So $\Phi$ is a $C^{*}$-isomorphism from
$\bgga$ into $\bga$.

For the last part of the theorem,
from Proposition~\ref{prop:span} and the density of $\gach$ in $\gah$, we just
have to show that $\Phi(T(e_{1},e_{2}))\in \kga$. In fact,
$\Phi(T(e_{1},e_{2}))$
is just the ``rank $1$'' operator $\theta_{e_{1},e_{2}}$
(\cite[p.128]{Blackadar}). Indeed, for $e\in \gach$ and $u\in G^{0}$,
using \ref{eq:te1e2'}:
\beqns
\Phi(T(e_{1},e_{2}))e(u)&=& T(e_{1},e_{2})^{u}e  \\
&=& \int L_{g}e_{1}\ov{\lan L_{g}e_{2},e\ran}\, d\la^{u}(g)   \\
&=& \int L_{g}e_{1}\ov{\lan e_{2},L_{g^{-1}}e\ran}\, d\la^{u}(g)   \\
&=& \int L_{g}e_{1}\lan e_{2},e\ran(g^{-1})\, d\la^{u}(g)   \\
&=& e_{1}\lan e_{2},e\ran(u)  \\
&=& \theta_{e_{1},e_{2}}(e)(u).
\eeqns
\end{proof}

We can now complete the construction of the analytic index of the
pseudodifferential operator $D$.  By Proposition~\ref{prop:parinv}, $D$
is an invariant Fredholm morphism on $\mfH$ and there exists an
$S\in \bgga$ such that both $SD-I,DS-I\in \kgga$.
Let $a=\Phi(D), b=\Phi(S)$. By Theorem~\ref{th:morK}, we have
$ab-I, ba-I\in \kga$.  We summarize the well-known construction of a Fredholm
module from $a$.  Let $\pi:\bga\to \bga/\kga$ be the quotient map.  Then
$\pi(a)$ is invertible.  Let $u$ be the unitary part of $\pi(a)$ given by the
polar decomposition of $\pi(a)$: so
$u=\pi(a)[\pi(a^{*}a)]^{-1/2}$. Let $c\in \bga$ be such that $\pi(c)=u$.  Then
$cc^{*}-I, c^{*}c-I\in \kga$.
We then get a Kasparov $(\C,C_{red}^{*}(G))$-bimodule
\[ (\gah\oplus \gah,
\normalsize
\begin{pmatrix}
0 & c^{*} \\
c & 0
\end{pmatrix}
\large
)
\]
which gives an element of $KK^{0}(\C, C_{red}^{*}(G))=
K_{0}(C_{red}^{*}(G))$.  This element is independent of the choice of $c$ by
the invariance of Kasparov classes under compact perturbations.
This element is the analytic index of $D$.



\end{document}